\documentclass[a4paper,reqno]{amsart}
\usepackage{amssymb}
\usepackage{tikz}
\usepackage{hyperref}
\usepackage{mathtools}
\mathtoolsset{showonlyrefs}

 \allowdisplaybreaks \numberwithin{equation}{section}
\begingroup
\theoremstyle{plain}
\newtheorem{theorem}{Theorem}[section]
\newtheorem{proposition}[theorem]{Proposition}
\newtheorem{lemma}[theorem]{Lemma}
\newtheorem{corollary}[theorem]{Corollary}

\theoremstyle{definition}

\newtheorem{remark}[theorem]{Remark}

\endgroup

\newcommand{\1}{ 1\!\!1}

\newcommand{\res}{\mathop{\hbox{\vrule height 7pt width .5pt depth
0pt\vrule height .5pt width 6pt depth0pt}}\nolimits}

\def \dist {\mathop {\rm dist}\nolimits}

\def \e {\epsilon}
\def \a {{a}}

\def \re {\mathbb R}

\def \R {\mathbb R}
\def \Om {\Omega}

\newcommand{\de}{\mathrm{d}}

\def \diam {\mathop{\rm diam}}

\newcommand{\wto}{\rightharpoonup}
\newcommand{\wsto}{\stackrel{*}\wto}
\newcommand{\norma}[1]{\lVert{#1}\rVert}

\newcommand{\mygraphic}[1]{\includegraphics[height=#1]{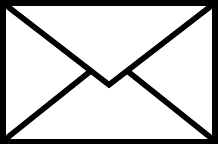}}
\newcommand{\myenv}{(\raisebox{0pt}{\mygraphic{.6em}})}

\newcommand{\cF}{\mathcal{F}}
\newcommand{\cH}{\mathcal{H}}

\newcommand{\cL}{\mathcal{L}}

\newcommand{\cP}{\mathcal{P}}

\newcommand{\B}{B}

\newcommand{\average}{{\mathchoice {\kern1ex\vcenter{\hrule
height.4pt width 8pt depth0pt}
\kern-11pt} {\kern1ex\vcenter{\hrule height.4pt width 4.3pt
depth0pt} \kern-7pt} {} {} }}

\newdimen\mex
\makeatletter
\def\niv{\mathrel{\hbox{\hglue -0.4\mex
\vrule \@height 1.4\mex \@width .14\mex
\vrule \@height .14\mex \@width .75\mex
\hglue -0.2\mex}}}
\makeatother

\title{Upscaling of screw dislocations with increasing tangential strain}

\author{Ilaria Lucardesi}
\address[I.\@ Lucardesi]{Institut \'Elie Cartan de Lorraine, B.P.\@ 70239, 54506 Vandoeuvre-l\`es-Nancy, France}
\email{ilaria.lucardesi@univ-lorraine.fr}
\author{Marco Morandotti}
\address[M.\@ Morandotti \myenv]{Fakult\"at f\"ur Mathematik, Technische Universtit\"at M\"unchen, Boltzmannstrasse, 3, 85748 Garching, Germany}
\email{marco.morandotti@ma.tum.de}
\author{Riccardo Scala}
\address[R.\@ Scala]{Dipartimento di Matematica ``G.~Castelnuovo'', La Sapienza Universit\`a di Roma, Piazzale Aldo Moro, 5, 00185 Roma, Italy}
\email{rscala@fc.ul.pt}
\author{Davide Zucco}
\address[D.\@ Zucco]{Dipartimento di Matematica ``G.~Peano'', Universit\`a di Torino, Via Carlo Alberto, 10, 10123 Torino, Italy}
\email{davide.zucco@unito.it}
\date{\today}

\begin{document}

\maketitle 

\begin{abstract}
The upscaling of a system of screw dislocations in a material subject to an external strain is studied.
The $\Gamma$-limit of a suitable rescaling of the renormalized energy is characterized in the space of probability measures.
This corresponds to a discrete-to-continuum limit of the dislocations, which, as a byproduct, provides information on their distribution when the circulation of the tangential component of the external strain becomes larger and larger.
In particular, dislocations are shown to concentrate at the boundary of the material and to distribute as the limiting external strain.
\end{abstract}

\smallskip



\section{Introduction}

Dislocations are line defects in the lattice structure of crystalline materials, that have been first observed in metals by electron microscopy in 1956 by Hirsch, Horne, and Whelan \cite{HHW}, thus providing experimental support to the theoretical work of Volterra \cite{volterra}.
The literature on dislocations is vast, including physical, engineering, and mathematical approaches; we refer the reader to the classical monographs \cite{HL,HB,nabarro} for general treaties.
From the mechanical point of view, dislocations are of the utmost importance to understand some properties of materials, especially those related to their plastic behaviour: in 1934, Orowan \cite{orowan}, Polanyi \cite{polanyi}, and Taylor \cite{taylor} theorized that dislocations are the ultimate cause of plasticity, thus proposing that the presence and motion of defects at the atomic scale is responsible for macroscopic effects.

It is therefore relevant to bridge phenomena happening at different length scales by means of a suitable limiting process which allows one to derive an averaged, macroscopic quantity from discrete, microscopic ones. A macroscopic strain gradient theory for plasticity from a model of discrete dislocations was obtained by Garroni, Leoni, and Ponsiglione \cite{GLP}.
Further and more recent results include \cite{ADLGP}, where the dynamics of topological singularities in two dimensions is studied and compared to that of Ginzburg-Landau vortices, and \cite{vMM2018}, where a discrete-to-continuum limit for particles with an annihilation rule in dimension one is obtained.

\smallskip

In this paper we focus our attention on an isotropic crystal which occupies a vertical cylindrical region $\Omega \times \mathbb R$ and which undergoes antiplane shear. In this case, the dislocation lines are parallel to the lattice mismatch, here vertical, and the dislocations are called of \emph{screw} type. According to the model proposed by Cermelli and Gurtin in \cite{CG}, the system is fully described by the cross section of the material. This allows us to work in $\Omega\subset\mathbb{R}^2$ so that  dislocations are represented by points $\{a_i\}_i\subset \Omega$, corresponding to the intersections of the dislocation lines with the cross section. Throughout the paper, without any further explicit reference, we assume $\Omega$ to be a bounded open domain with Lipschitz boundary and the lattice spacing of the material to be $2\pi$. 

The stressed material is described by the \emph{strain field}, a vector field with curl concentrated on the discrete set of dislocations $\{a_i\}_i$, and minimizing the energy
$$
F\mapsto \int_\Omega |F(x)|^2\, \de x.
$$
Due to the singularity of the curl, such energy is not finite. The usual strategy to circumvent this obstruction consists in avoiding the dislocation \emph{cores}, small disks $\{B_\e(a_i)\}_i$ around each dislocation $a_i$ with radius $\e>0$ sufficiently small, and then computing the energy in the resulting perforated domain,
i.e.,
\begin{equation*}
\mathcal E_\e(a_1,\dots,a_n)\coloneqq\min \frac{1}{2} \,\int_{\Omega\setminus \bigcup_i \overline B_\epsilon(a_i)} |F(x)|^2\,\de x,
\end{equation*}
where the minimum is taken among all vector fields in $L^2(\Omega\setminus \bigcup_i \overline B_\epsilon(a_i); \re^2)$, with zero curl in the sense of distributions $\mathcal D'\left(\Omega\setminus \bigcup_i \overline B_\epsilon(a_i)\right)$, and satisfying the condition $\langle F\cdot \tau,1 \rangle_\gamma=2\pi m$ for an arbitrary simple closed curve $\gamma$ in $\Omega\setminus \bigcup_i \overline B_\epsilon(a_i)$ winding once counterclockwise around $m$ dislocations. Here and henceforth  $\tau$ denotes the tangent unit vector to $\gamma$, $F\cdot \tau$ must be intended in the sense of traces, and $\langle \cdot ,\cdot \rangle_\gamma$ denotes the pairing between $H^{-1/2}(\gamma)$ and $H^{1/2}(\gamma)$ (see \cite{ChenFrid}). 

Once the non-integrability of the strain field around the dislocations is removed, classical variational techniques can be applied. For this reason, the  \emph{core radius approach} is standard in the literature and it is employed in different contexts, such as linear elasticity (see, for instance, \cite{nabarro,TOP,VKLLO}; also \cite{ADLGP,BM,CG,pons} for screw dislocations and \cite{CL} for edge dislocations), the theory of Ginzburg-Landau vortices (see, for instance, \cite{BBH,SaSe} and the refereces therein), and liquid crystals (see, for instance, \cite{GSV}).  

In the setting above, the dynamics of dislocations is determined by the Neumann boundary condition satisfied by the strain field (the natural boundary conditions coming from the Euler-Lagrange equation associated to the energy above), implying that the energy of the system decreases as the dislocations approach the boundary of the domain. In other words, as time passes, 
dislocations move towards the boundary and leave the domain in finite time (see \cite{BFLM,HM}). 

Different types of boundary condition can be imposed if one is interested in the confinement of the dislocations inside the material.
We impose a Dirichlet boundary condition for the strain field: physically, this corresponds to stressing the material by means of a prescribed external strain (see \cite{BBH, LMSZ, SS00}).
More specifically, given $n\in \mathbb N$ we prescribe the tangential strain on the boundary of the exterior domain $\partial \Om$ to be a distribution $f^n\in H^{-1/2}(\partial \Om)$ such that $\langle f^n,1\rangle_{\partial\Omega}=2\pi n$.
Then, in the minimization problem above, we also require that $F\cdot \tau=f^n$ on $\partial \Om\setminus \bigcup_{i=1}^n \overline{B}_\e(a_i)$, where $\tau$ denotes the tangent unit vector to $\partial \Omega$ (which exists $\cH^1$-a.~e.~on $\partial \Omega$ thanks to the regularity assumption on $\Omega$). 
This Dirichlet boundary condition reflects in \emph{confinement and separation effects}: for every $\epsilon$ sufficiently small one can observe $n$ \emph{distinct} dislocations $(a_1,\ldots,a_n)$ \emph{inside} $\Om$ (see, e.g., \cite{BBH,SaSe} for a comment on the topological necessity of the presence of exactly $n$ defects). Indeed, for $\epsilon$ sufficiently small, it can be shown (see \cite{LMSZ} and also \cite{BBH,SaSe}) that the energy behaves like
\[
\mathcal E_\e(a_1,\dots, a_n)=-\pi n\log \epsilon+ \mathcal E_0(a_1, \dots a_n)+o(1),
\] 
which is infinity whenever one dislocation is on the boundary $\partial \Omega$ or when two dislocations collide.
The \emph{renormalized energy}  $\mathcal E_0$ is also related to the so-called \textit{Hadamard finite part} of a divergent integral (see \cite{Had}) and 
keeps into account the energetic dependence of the position of the dislocations $\{a_i\}_i$ inside $\Omega$. 
One has in particular that
\begin{equation}\label{recast}
\mathcal E_0(a_1, \dots, a_n)=\sum_{i=1}^n\mathcal E_\text{self}(a_i)+ \sum_{i\neq j}\mathcal E_\text{int}(a_i,a_j)
\end{equation}
where the \emph{self energy}  $\mathcal E_\text{self}$, responsible for the contribution of individual dislocations, is given by 
\[
\mathcal E_\text{self}(a_i) \coloneqq \pi\log d(a_i) +\frac{1}{2}\int_{\Om\setminus \overline{B}_{d(\a_i)}(\a_i)}|\nabla \phi_i(x) + \nabla w_i(x)|^2\, \de x+\frac{1}{2}\int_{B_{d(\a_i)}(\a_i)}|\nabla w_i(x)|^2\,\de x,
\]
with $d(a_i)\coloneqq \mathrm{dist}(\a_i, \partial \Omega)$, while the \emph{interaction energy} $\mathcal E_\text{int}$, which encodes the mutual position of two dislocations, is 
\[
\mathcal E_\text{int}(a_i,a_j) \coloneqq \int_\Omega (\nabla \phi_i(x)+\nabla w_i(x))\cdot(\nabla \phi_j(x)+\nabla w_j(x))\,\de x.
\] 
Here the functions $\phi_i$ and $w_i$ are the solutions ($w_i$ is defined up to an additive constant) to
\begin{equation}\label{M2}
\begin{cases}
\Delta \phi_i=2\pi\delta_{a_i} & \text{in $\Omega$,} \\
\phi_i(x)=\log|x-a_i| & \text{for $x\in\partial\Omega$,}
\end{cases}
\qquad\text{and}\qquad
\begin{cases}
\Delta w_i=0 & \text{in $\Omega$,} \\
\partial_\nu w_i=\tfrac1n f^n-\partial_\nu\phi_i & \text{on $\partial\Omega$.}
\end{cases}
\end{equation}
Notice that the solution $\phi_i$, which in principle is searched only in $\Omega$, is well defined also in the complement $\mathbb R^2\setminus \overline \Omega$: clearly its explicit formula is $\phi_i(x)=\log|x-a_i|$ so that, given a regular bounded domain $B$ containing $\overline \Omega$, we have $\phi_i\in W^{1,p}(B)$ for every $1\leq p<2$ (the value $p=2$ is excluded, due to the singularity of the logarithm).
Without loss of generality, we always consider $w_i$ as the solution with zero average in $\Omega$. 

Apart from the following existence result, that has been proved in \cite{LMSZ} (see also \cite{BBH, SaSe} for related results for Ginzburg-Landau vortices), nothing is known on the positions of the dislocations minimizing the energy functional \eqref{recast}.
\begin{theorem}[Existence]
Given $n\in\mathbb N$, the minimization problem
\begin{equation}\label{probmin}
\min \mathcal E_0(a_1,\dots, a_n)
\end{equation}
among all $n$-tuples of distinct points in $\Omega$ has a solution.
\end{theorem}
Note that the lack of compactness of the underlying space where the minimization problem is set makes the result highly non trivial. 
In this paper we address the \emph{upscaling} problem:  we study the asymptotic distribution in $\overline \Omega$, as $n\to\infty$, of the minimizing configurations of the renormalized energy $\mathcal E_0$. 
As the number of dislocations increases, it is not practical anymore to describe them as individual particles, but it is necessary to associate them a probability measure describing their distribution in $\overline \Omega$. This is usually achieved (see, e.g,
\cite{BJM2011, vMM2018,MPS2017, TZ2018}) by considering the \emph{empirical measures}
\begin{equation}\label{average}
\mu^n\coloneqq\frac1n\sum_{i=1}^n\delta_{a_i},
\end{equation}
with $\delta_{a_i}$ denoting the Dirac measure centered at $a_i$, and studying the $\Gamma$-limit of the sequence of functionals $\cF^n\colon\cP(\overline\Omega)\to\R\cup\{+\infty\}$ defined by
\begin{equation}\label{T9}
\cF^n(\mu^n)\coloneqq\begin{cases}
\displaystyle \frac{1}{n^2}\mathcal E_0(a_1,\dots, a_n) & \text{if $\mu^n$ is of the form \eqref{average},} \\
+\infty & \text{otherwise.}
\end{cases}
\end{equation}
Notice that the rescaling by $1/n^2$ does not affect the solution to the minimization problem \eqref{probmin}.
Moreover, it is the natural one in order to prevent the renormalized energy to diverge in the limit when $n\to\infty$. This is evident from the expression \eqref{recast} of the energy $\mathcal{E}_0$, since the contribution $\mathcal{E}_{\mathrm{self}}$ is the sum of $n$ quadratic terms and the contribution $\mathcal{E}_{\mathrm{int}}$ involves $n\times n$ pairwise interactions.

We are now ready to state the main result of the paper.
\begin{theorem}[Upscaling]\label{L1}
Let $n\in \mathbb N$ and $f^n\in H^{-1/2}(\partial \Om)$ with $\langle f^n,1\rangle_{\partial \Omega}=2\pi n$. Assume that $\tfrac1n f^n\to f$ strongly in $H^{-1/2}(\partial \Omega)$ as $n\to\infty$. Then the functionals $\cF^n$ in \eqref{T9} $\Gamma$-converge, with respect to the weak-* convergence in $\cP(\overline\Omega)$, as $n\to\infty$, to the functional $\cF^\infty\colon\cP(\overline\Omega)\to[0,+\infty]$ defined by
\begin{equation}\label{T1}
\cF^\infty(\mu)\coloneqq\begin{cases}
\displaystyle \frac12\int_\Omega |\nabla U(x)|^2\,\de x  & \text{if $\mu\in H^{-1}(B)$,} \\
+\infty & \text{otherwise},
\end{cases}
\end{equation}
where $B$ is any regular bounded domain containing $\overline\Omega$ and
the function $U\in H^1(\Omega)$ solves
\begin{equation}\label{T2}
\begin{cases}
\Delta U = 2\pi\mu & \text{in $\Om$,} \\
\partial_\nu U = f -2\pi\mu\res\partial \Om& \text{on $\partial\Omega$,}
\end{cases}
\end{equation}
in the sense that
$$\int_\Omega |\nabla U(x)|^2\,\de x=
2\pi\langle f,U\rangle_{\partial\Omega}-2\pi \int_{\overline\Omega} U(x)\,\de\mu(x).$$
\end{theorem}

\begin{remark}\label{labox}
We notice that the energy $\mathcal F^\infty$ defined in \eqref{T1} is independent of the set $B$. This is trivial, since the support of $\mu$ is contained in $\overline\Omega$.
Therefore, in the sequel we choose a particular regular bounded domain $B$ such that $\overline\Omega\subset B$ and we call it \emph{box}.
\end{remark}

\begin{remark}\label{allafinelocitiamo}
Notice that problem \eqref{T2} is well defined in $H^1(\Omega)$, since the boundary datum $f -2\pi\mu\res\partial \Om$ belongs to $H^{-1/2}(\partial \Om)$. This fact follows from the assumption on $f$ and by the fact that $\mu\res\partial \Om$ is the jump of the normal derivative across $\partial \Omega$ of an $H^1$ function. More precisely,  
\begin{equation}\label{ferragosto}
2\pi \mu\res\partial \Om=\partial_\nu\phi_+-\partial_\nu\phi_-,
\end{equation}
with 
\begin{equation}\label{lorecolliamo}
\phi(x) \coloneqq \int_{\R^2}\log|x-y|\,\de\mu(y).
\end{equation}
Here and henceforth $\partial_\nu\phi_+$ and $\partial_\nu\phi_-$ denote the normal traces of $\nabla \phi$ from the \emph{exterior} and from the \emph{interior} of $\Omega$, respectively, and are defined in the natural way by duality: for any $\psi \in H^1(B)$, we set
\begin{subequations}
\begin{align}
& \langle \partial_\nu \phi_+,\psi \rangle_{\partial \Om}\coloneqq -\int_{\B \setminus \overline \Omega} \nabla \phi \cdot \nabla \psi\, \de x+\int_{\partial \B}\psi\, \partial_\nu \phi  \,\de\cH^1,\label{tracciaesterna}
\\
&  \langle \partial_\nu \phi_-, \psi \rangle_{\partial \Om}\coloneqq 2\pi\int_\Omega \psi \, \de \mu + \int_\Omega \nabla \phi \cdot \nabla \psi\, \de x ,\label{tracciainterna}
\end{align}
\end{subequations}
where we have taken into account that $\Delta \phi=2\pi \mu$ and that $\mu$ has support in $\overline \Omega$.
Notice that in the right-hand side of \eqref{tracciaesterna} we have replaced the pairing $\langle \partial_\nu \phi,\psi \rangle_{\partial B}$ by the integral over $\partial B$, since $\phi$ is regular on $\partial B$. 
Clearly, if $\phi$ were regular also on $\partial \Omega$, then both traces would coincide with the standard normal derivative.
Taking the difference of \eqref{tracciaesterna} and \eqref{tracciainterna}, by using the divergence theorem and the fact that $\mu$ is concentrated in $\overline \Omega$, we obtain
\begin{align*}
\langle  \partial_\nu \phi_+ - \partial_\nu \phi_-, \psi \rangle_{\partial \Om}& = 
-\int_{\B} \nabla \phi \cdot \nabla \psi\, \de x
 +\int_{\partial \B}\psi \, \partial_\nu \phi  \,\de\cH^1 -2\pi \int_\Omega \psi \, \de \mu
 \\ & =2\pi \int_B \psi \, \de \mu -2\pi \int_\Omega \psi \, \de \mu = 2\pi\langle \mu\res\partial \Om, \psi\rangle_{\partial \Omega}\,,
\end{align*} 
which gives \eqref{ferragosto}, by the arbitrariness of $\psi$.
\end{remark}

As a consequence of Theorem \ref{L1} (see \cite{dalmaso}), since $\cP(\overline\Omega)$ is compact with respect to the weak-* topology, we deduce the following result.
\begin{corollary}\label{L2}
Let $n\in \mathbb N$ and $f^n\in H^{-1/2}(\partial \Om)$ with $\langle f^n,1\rangle_{\partial\Omega}=2\pi n$. Assume that $\tfrac1n f^n\to f$ strongly in $H^{-1/2}(\partial \Omega)$ as $n\to\infty$. If $(a_1,\dots, a_n)$ is a minimizer of \eqref{probmin}, the corresponding empirical measures $\mu^n$ defined by \eqref{average} weak-* converge to $\mu^\infty$, where $\mu^\infty\in\cP(\overline\Omega)$ is the unique minimizer of the functional $\cF^\infty$ defined in \eqref{T1}. Moreover, $\cF^n(\mu^n)\to\cF^\infty(\mu^\infty)$, as $n\to\infty$.
\end{corollary}

For particular choices of $f$ we are able to characterize explicitly the measure $\mu^\infty$ and then to derive information on minima and minimizers of \eqref{probmin}.
\begin{corollary}\label{L3}
Let $n\in \mathbb N$ and $f^n\in H^{-1/2}(\partial \Om)$ with $\langle f^n,1\rangle_{\partial\Omega}=2\pi n$. Assume that $\tfrac1n f^n\to f$ strongly in $H^{-1/2}(\partial \Omega)$ as $n\to\infty$. If $f\in L^1(\partial \Omega)$ and $f\geq 0$ then if $(a_1,\dots, a_n)$ is a minimizer of \eqref{probmin}, the corresponding empirical measures $\mu^n$ defined by \eqref{average} weak-* converge to the measure $\mu^\infty$, which is absolutely continuous with respect to $\cH^1$, and is defined by 
\begin{equation}\label{limite}
\mu^\infty=\frac{f}{2\pi}\,\cH^1\res\partial \Omega.
\end{equation}
Moreover, the energy vanishes in the limit, i.e.,
\begin{equation}\label{vanishinglimit}
\lim_{n\to\infty}\cF^n(\mu^n)=0.
\end{equation}
\end{corollary}

A related result to Theorem~\ref{L1} can be found in the framework of Ginzburg-Landau theory \cite{SS00}, where the authors treated only the case $f^n= n f$, for a \emph{fixed continuous} function $f$, independent of $n$. 
We point out that the interest in studying more general sequences $f^n$ of \emph{non constant} and \emph{non regular} boundary data has been raised by Sandier and Soret in \cite[Open Problem 1]{SS2000}. In light of these similarities, our results can be regarded as generalizations of those contained in the papers \cite{SS2000,SS00}. We underline that we also weaken the assumptions on the regularity of the domain $\Omega$ and we do not require simple connectedness. 
For these reasons, the proof strategy of \cite{SS2000,SS00}, based on the regularity of the domain and the boundary datum, does not seem easily adaptable to the present case. In this respect, the introduction of the box $B$ will be crucial in the proof of the $\Gamma$-convergence result to recast the renormalized energy \eqref{T9} and pass to the limit in the annular region $\B\setminus \overline \Omega$.

The starting point of our analysis is the rewriting of the energy $\mathcal F^n$ in Section~\ref{sec.2}. Then, in Section~\ref{sec.3} we prove some auxiliary lemmas for the proof of Theorem~\ref{L1}, which is finally addressed in Section~\ref{sec.4}, togheter with the proofs of Corollaries~\ref{L2} and \ref{L3}.

\subsection*{Notation}
The symbol $\cH^1$ denotes the $1$-dimensional Hausdorff measure, while $\sharp$ the counting measure.
Given an open set with Lipschitz boundary, we denote by $\nu$ the outer unit normal vector to the boundary, defined almost everywhere on it. We always use the symbols $\sum_{i\neq j}$ and $\sum_{i<j}$ to denote the summation over all indices $i,j$ with $i\neq j$ and $i<j$, respectively, ranging from $1$ to some $n\in \mathbb N$ whose value is clear from the context. 
Recalling Remark~\ref{labox}, we fix once and for all a box $B$ containing $\overline \Omega$ and we denote by $\triangle_0=\{(y,z)\in\B\times\B:y=z\}$ its diagonal. 
Given $x\in \mathbb R^2$ and $r\in \mathbb R^+$, we denote by $B_r(x)$ the open disk centered at $x$ with radius $r$, and by $\overline{B}_r(x)$ its closure.
The duality product between $H^{-1/2}(\partial \Om)$ and $H^{1/2}(\partial \Om)$ is denoted by $\langle\cdot,\cdot\rangle_{\partial \Om}$.
Given $A\subset \mathbb R^2$, we denote by $\1_A$ its characteristic function, namely $\1_A(x)=1$ if $x\in A$ and $\1_A(x)=0$ if $x\notin A$.

\section{The renormalized energy}\label{sec.2}

In this section we rewrite the energy functional $\mathcal F^n$ introduced in \eqref{T9} in a more convenient way. Let $\mu^n\in \mathcal P(\overline\Omega)$ be the empirical measure associated to an $n$-tuple $(a_1,\dots, a_n)$ of distinct points $a_i \in \Omega$ according to \eqref{average}. We define
\begin{equation}\label{M3}
\phi^n \coloneqq \frac1n\sum_{i=1}^n \phi_i\qquad\text{and}\qquad w^n \coloneqq \frac1n\sum_{i=1}^n w_i.
\end{equation}
In view of \eqref{M2}, these functions solve, respectively, 
\begin{equation}\label{M4}
\begin{cases}
\Delta \phi^n=2\pi\mu^n & \text{in $\Omega$,} \\
\phi^n(x)=\frac1n\sum_{i=1}^n\log|x-a_i| & \text{on $\partial\Omega$,}
\end{cases}
\qquad\text{and}\qquad
\begin{cases}
\Delta w^n=0 & \text{in $\Omega$,} \\
\partial_\nu w^n(x)= \tfrac 1n f^n-\partial_\nu\phi^n & \text{on $\partial\Omega$.}
\end{cases}
\end{equation}
Notice that $\phi^n$ belongs to $W^{1,p}(B)$ for every $1\leq p<2$ (but not to $H^1(B))$ and reads
$$
\phi^n(x)=\int_\Omega \log|x-y| \,\de\mu^n(y).
$$ 
Moreover, since all the $a_i$'s are inside $\Omega$, the \emph{exterior} and \emph{interior} normal traces of $\phi^n$ coincide, namely $\partial_\nu\phi^n_+=\partial_\nu\phi^n_-$ (and the same holds for all the $\phi_i$'s as well, see Remark~\ref{allafinelocitiamo}).
As for $w^n$, since every $w_i$ has zero average, it is uniquely determined and has zero average in $\Omega$. 
Performing an integration by parts in \eqref{recast} and exploiting \eqref{M2}, we get
\begin{align*}
\mathcal E_\text{self}(a_i)& = \pi \log d(a_i) + \frac12 \int_{\Omega\setminus \overline{B}_{d(a_i)}(a_i)}\big(|\nabla \phi_i|^2 + 2 \nabla \phi_i \cdot \nabla w_i\big)\, \de x + \frac12 \int_\Omega |\nabla w_i|^2\, \de x
\\
& =-  \frac12 \langle \partial_\nu\phi_i,\phi_i\rangle_{\partial \Omega} + \frac 1n \langle  f^n, \phi_i \rangle_{\partial \Omega} + \frac12  \int_\Omega |\nabla w_i|^2\, \de x,
\end{align*}
so that the contribution of the self energy in \eqref{T9} can be written as
\begin{equation}\label{M11}
\frac1{n^2} \sum_{i=1}^n \mathcal E_\text{self}(a_i) = -\frac1{2n^2}\sum_{i=1}^n \langle \partial_\nu\phi_i,\phi_i\rangle_{\partial \Omega} +\frac{1}{n^2}\langle {f^n}, \phi^n  \rangle_{\partial\Omega}  +\frac1{2n^2} \sum_{i=1}^n \int_\Omega |\nabla w_i|^2\, \de x.
\end{equation}
The terms in the interaction energy $\mathcal E_\text{int}(a_i,a_j)$, after expanding the power and integrating by parts, read
\begin{subequations}\label{T3}
\begin{align}
\frac1{2n^2} \sum_{i\neq j}\int_\Omega \nabla\phi_i\cdot\nabla\phi_j \, \de x={} & \frac1{2n^2}\sum_{i\neq j} \langle \partial_\nu\phi_j,\phi_i\rangle_{\partial \Omega}-\pi\int_{\Omega\times\Omega} \log|x-y|\,\de\mu^n\boxtimes\mu^n, \label{M6} \\
\frac1{2n^2} \sum_{i\neq j}\int_\Omega \nabla\phi_i\cdot\nabla w_j \, \de x={} & \frac1{2n^2}\sum_{i\neq j} \langle \tfrac 1n {f^n}-\partial_\nu \phi_j,  \phi_i\rangle_{\partial\Omega}, \label{M7} \\
\frac1{2n^2} \sum_{i\neq j}\int_\Omega \nabla\phi_j\cdot\nabla w_i \, \de x={} & \frac1{2n^2}\sum_{i\neq j} \langle \tfrac 1n {f^n}-\partial_\nu \phi_i,  \phi_j\rangle_{\partial\Omega}, \label{M8} \\
\frac1{2n^2} \sum_{i\neq j}\int_\Omega \nabla w_i\cdot\nabla w_j \, \de x={} & \frac1{2}\int_{\Omega} |\nabla w^n|^2\de x-\frac1{2n^2}\sum_{i=1}^n \int_\Omega |\nabla w_i|^2\de x, \label{M9}
\end{align}
\end{subequations}
where we set
\begin{equation}\label{box}
\mu^n\boxtimes\mu^n \coloneqq \frac1{n^2}\sum_{i\neq j}\delta_{(a_i,a_j)}.
\end{equation}
By grouping \eqref{M6}, \eqref{M7}, and \eqref{M8} together, recalling \eqref{M4} and \eqref{box}, we obtain the contribution
\[
\begin{split}
&\frac1{2n^2} \sum_{i\neq j}\left(\int_\Omega \nabla\phi_i\cdot\nabla\phi_j \, \de x +\int_\Omega \nabla\phi_i\cdot\nabla w_j \, \de x+\int_\Omega \nabla\phi_j\cdot\nabla w_i \, \de x\right)\\
&=\frac{n-1}{n^2}\langle f^n,\phi^n\rangle_{\partial \Omega} -\frac12\langle \partial_\nu\phi^n,\phi^n\rangle_{\partial \Omega}+\frac1{2n^2} \sum_{i=1}^n\langle \partial_\nu\phi_i,\phi_i\rangle_{\partial \Omega} -\pi\int_{\Omega\times\Omega} \log|x-y|\,\de\mu^n\boxtimes\mu^n,
\end{split}
\]
which combined with \eqref{M9} yields
\[
\begin{split}
\frac{1}{n^2}\sum_{i\neq j} \mathcal E_\text{int}(a_i,a_j)={}&\frac{n-1}{n^2}\langle  {f^n},\phi^n\rangle_{\partial \Omega} -\frac12\langle \partial_\nu\phi^n,\phi^n\rangle_{\partial \Omega}+\frac1{2n^2} \sum_{i=1}^n \langle \partial_\nu\phi_i,\phi_i\rangle_{\partial \Omega}\\
&
-\pi\int_{\Omega\times\Omega} \log|x-y|\,\de\mu^n\boxtimes\mu^n+\frac1{2}\int_{\Omega} |\nabla w^n|^2\de x-\frac1{2n^2}\sum_{i=1}^n \int_\Omega |\nabla w_i|^2\de x.
\end{split}
\]
Plugging this new expression together with \eqref{M11} into \eqref{T9}  allows to rewrite the renormalized energy functional as 
\begin{equation}\label{M5}
\cF^n(\mu^n) \!= \! \frac1{2}\int_{\Omega} |\nabla w^n|^2\,\de x 
+\frac 1n \langle  {f^n},\phi^n\rangle_{\partial \Omega}
- \frac12 \langle \partial_\nu\phi^n,\phi^n\rangle_{\partial \Omega}
-\pi\int_{\Omega\times\Omega} \log|x-y|\,\de\mu^n\boxtimes\mu^n,
\end{equation}
{which will be fundamental in the sequel.}

\section{Auxiliary lemmas}\label{sec.3}
In the following, we prove some auxiliary lemmas which will play a crucial role in the proof of the $\Gamma$-liminf inequality for Theorem \ref{L1}. The first one concerns an equi-coercivity property of the functionals $\cF^n$.

\begin{lemma}\label{lemma-equic}
For every $n\in\mathbb N$, let $\mu^n$ be as in \eqref{average}, let $f^n\in H^{-1/2}(\partial\Om)$ with $\langle f^n,1\rangle_{\partial \Om}=2\pi n$, and assume $\frac{1}{n} f^n$ be uniformly bounded in $H^{-1/2}(\partial\Om)$. Let $\phi^n$ and $w^n$ be as in \eqref{M3}.
Then there exist two constants $C_1,C_2>0$ independent of $n$ such that
\begin{equation}\label{equic}
\cF^n(\mu^n) \geq C_1\left( \norma{w^n}_{H^1(\Omega)}^2+\norma{\phi^n}_{H^1(\B\setminus\overline\Omega)}^2 \right)-\pi\int_{\Omega\times \Omega} \log|x-y|\,\de\mu^n\boxtimes\mu^n (x,y)
-C_2.
\end{equation}
\end{lemma}
{
\begin{proof} By hypothesis, $\cF^n(\mu^n)$ is of the form \eqref{M5}. We have assumed that each $w_i$ in \eqref{M2}, and hence $w^n$ in \eqref{M3}, have zero average in $\Omega$, thus, by the Poincar\'e-Wirtinger inequality, the first term in the right-hand side of \eqref{M5} is equivalent to the $H^1$ norm of $w^n$. 
As for $\phi^n$, we note that $\|\phi^n\|^2_{L^2(B\setminus \overline \Omega)}$ is uniformly bounded:
\begin{equation*}
\begin{split}
\|\phi^n\|^2_{L^2(B\setminus \overline \Omega)} & = \int_{B\setminus \overline\Omega} \bigg|\frac 1n \sum_{i=1}^n \phi_i\bigg|^2 \de x = \frac{1}{n^2} \int_{B\setminus \overline\Omega} \bigg( \sum_{i=1}^n |\phi_i|^2 + 2\sum_{i\neq j} \phi_i \phi_j \bigg) \, \de x
\\
& 
\leq \frac{2n + 1}{n^2} \sum_{i=1}^n \int_{B \setminus \overline \Omega} \|\phi_i\|^2_{L^2(B)}\,\de x \leq C,
\end{split}
\end{equation*}
for some positive constant $C$ independent of $n$. In particular, we get
\begin{equation}\label{17}
\norma{\nabla \phi^n}_{L^2(\B\setminus\overline\Omega ;\mathbb R^2 )}^2 \geq  \norma{ \phi^n}_{H^1(\B\setminus\overline\Omega)}^2- C.
\end{equation}
By using the divergence theorem in $\B\setminus\overline\Omega$, we can write the third term in the right-hand side of \eqref{M5} as
\begin{equation}\label{T10}
- \langle \partial_\nu\phi^n,\phi^n\rangle_{\partial \Omega}=\int_{\B\setminus\overline\Omega} |\nabla \phi^n|^2\,\de x- \int_{\partial\B} \phi^n \partial_\nu\phi^n\,\de\cH^1.
\end{equation}
Here the regularity of $\phi^n$ and $\partial_\nu \phi^n$ on $\partial B$ (in particular the fact that $\partial B$ is disjoint from $\overline \Omega$, where the measure $\mu^n$ concentrates) allows us to replace the duality brackets $\langle \partial_\nu \phi^n, \phi^n \rangle_{\partial B}$ with an integral over $\partial B$.
By combining \eqref{T10}, \eqref{17}, and the assumption on the uniform boundedness of $\frac 1n f^{n}$, we obtain
\begin{align}
\frac 1n \langle  f^n,\phi^n\rangle_{\partial \Om}- \frac12\langle \partial_\nu\phi^n,\phi^n\rangle_{\partial \Omega} & \geq {\frac12} \norma{ \phi^n}_{H^1(\B\setminus\overline\Omega)}^2 -C\norma{\tfrac 1n f^n}_{H^{-1/2}(\partial\Omega)}\norma{\phi^n}_{H^1(\B\setminus\overline\Omega)} -C_2 \notag
\\ & 
\geq C_1  \norma{ \phi^n}_{H^1(\B\setminus\overline\Omega)}^2 - C_2, \label{T11}
\end{align}
where $C,C_1,C_2>0$ are three constants independent of $n$, which may vary from line to line. In the first inequality of \eqref{T11} we have used the trace theorem in $\B\setminus\overline\Omega$ and the fact that $\phi^n$ and $\partial_\nu\phi^n$ are uniformly bounded in $C^\infty(\partial\B)$; while in the second inequality we have used the Young's inequality $\norma{\tfrac 1n f^n}_{H^{-1/2}(\partial\Omega)}\norma{\phi^n}_{H^1(\B\setminus\overline\Omega)} \leq \tfrac{1}{2c} \norma{\tfrac 1n f^n}_{H^{-1/2}(\partial\Omega)}^2 + \frac{c}{2}\norma{\phi^n}_{H^1(\B\setminus\overline\Omega)}^2$ with a suitable choice of $c>0$, small enough.
The lemma is proved, possibly by changing the constants $C_1$ and $C_2$.
\end{proof}
}
As noticed in  \cite[Lemma 1]{GPPS2013} omitting the diagonal in the definition \eqref{box} of the measures $\mu^n\boxtimes\mu^n$
does not change their limiting behavior.

\begin{lemma}\label{nodiagonale}
Let $\mu^n,\mu\in\cP(\overline\Omega)$ be such that $\mu^n$ is of the form \eqref{average} for all $n\in\mathbb N$ and $\mu^n\wsto\mu$ as $n\to\infty$.
Then, $\mu^n\boxtimes\mu^n\wsto\mu\otimes\mu$.
\end{lemma}
\begin{proof}
Let $\psi\in C(\overline \Omega\times \overline \Omega)$ and write
\[
\int_{\overline \Omega\times \overline \Omega} \psi\,\de\mu^n\boxtimes\mu^n -\int_{\overline \Omega\times \overline \Omega} \psi\,\de\mu\otimes\mu=\int_{\overline \Omega\times \overline \Omega} \psi\,(\de\mu^n\boxtimes\mu^n -\de\mu^n\otimes\mu^n)+\int_{\overline \Omega\times \overline \Omega} \psi\,(\de\mu^n\otimes\mu^n-\de\mu\otimes\mu).
\]
Let us study the asymptotics  as $n\to \infty$ of the two terms in the right-hand side: the modulus of the first one is bounded above by $\|\psi\|_\infty/n$, thus it converges to zero; as for the second term, it vanishes thanks to the weak-* convergence of $\mu^n$ to $\mu$. This concludes the proof, by definition of weak-* convergence and by the arbitrariness of the continuous test function $\psi$.
\end{proof}

The previous two lemmas allows to transfer equiboundedness of the functionals $\cF^n$ into information on the measure $\mu\otimes \mu$.

\begin{lemma}\label{fatto2}
Let $\mu^n,\mu\in\cP(\overline\Omega)$ be such that $\mu^n\wsto\mu$ as $n\to\infty$.
Let us assume that $\cF^n(\mu^n)$ is uniformly bounded.
Then the measure $\mu\otimes\mu$ does not charge the diagonal $\triangle_0$. 
In particular, $\mu$ does not charge points.
\end{lemma}
\begin{proof}
Let $\{a_i^n\}$ be the family of points defining the measure $\mu^n$ in \eqref{average} and consider, for every $\e\in (0,1)$ and $n\in\mathbb N$, the quantity
\begin{equation*}
\begin{split}
N_{n,\e}\coloneqq{} &  \sharp\{(a_i^n,a_j^n)\in\B\times\B : a_i^n\neq a_j^n\;\text{and}\;|a_i^n-a_j^n|<\e\} \\
={} & \sharp\{(i,j)\in\{1,\ldots,n\}^2:i\neq j\;\text{and}\;(a_i^n,a_j^n)\in \triangle_\e\},
\end{split}
\end{equation*}
where $\triangle_\e$ denotes the $\e$-neighborhood of the diagonal of $\B\times\B$, namely the open set $
\triangle_\e\coloneqq\{(y,z)\in\B\times\B:|y-z|< \e\}$. By Lemma~\ref{lemma-equic}, recalling \eqref{box} and using the monotonicity of the logarithm, we have
\begin{equation}\label{aboveinequality}
\cF^n(\mu^n)\geq -\pi\int_{\triangle_\epsilon} \log|x-y|\,\de\mu^n\boxtimes\mu^n-C
\geq-\frac{\pi N_{n,\e}\log\e}{n^2}-C.
\end{equation}
By Lemma~\ref{nodiagonale}, the weak-* convergence $\mu^n\boxtimes\mu^n\wsto \mu\otimes \mu$ implies
\begin{equation*}
\liminf_{n\to\infty} \frac{N_{n,\e}}{n^2}=\liminf_{n\to\infty} \mu^n\boxtimes\mu^n(\triangle_\e)\geq 
\mu\otimes\mu(\triangle_\e)\geq \mu\otimes\mu(\triangle_0).
\end{equation*}
Thus, taking the $\liminf$ as $n\to \infty$ in \eqref{aboveinequality} and recalling that for $\e$ small $-\log \e >0$, we get
\[
\liminf_{n\to\infty}\cF^n(\mu^n)\geq -\pi \big(\mu\otimes\mu(\triangle_0)\big)\log \epsilon-C.
\]
Finally, the arbitrariness of $\epsilon\in (0,1)$ and the uniform boundedness of $\cF^n(\mu^n)$ imply that $\mu\otimes\mu(\triangle_0)=0$. The lemma is proved.
\end{proof}

We now prove a stability result for the functions $\phi^n$ introduced in \eqref{M3}.

\begin{lemma}\label{T13}
Let $\mu^n,\mu\in\cP(\overline\Omega)$ be such that $\mu^n$ is of the form \eqref{average} for all $n\in\mathbb N$, $\mu$ does not charge points, and $\mu^n\wsto\mu$ as $n\to\infty$.
Then the sequence of functions $\phi^n$ defined in \eqref{M3} converges strongly in $L^1(B)$ to 
the function
\begin{equation}\label{T14}
\phi(x) \coloneqq \int_{\overline\Omega} \log|x-y|\,\de\mu(y).
\end{equation} 
\end{lemma}
\begin{proof}
First, we notice that the function $\phi$ in \eqref{T14} belongs to $L^1(B)$, since it can be written as the convolution of an $L^1$ function with a probability measure (recall that $\mu$ is concentrated on $\overline \Omega$, thus the domain of integration $\overline\Omega$ can be replaced by the entire plane $\mathbb R^2$).
Let $M>0$ and let us consider the truncated functions
\begin{equation*}
\phi_M(x) \coloneqq \int_{\overline\Omega} \log|x-y|\vee(-M)\,\de\mu(y)\quad\text{and}\quad \phi_M^n(x) \coloneqq \int_{\overline\Omega} \log|x-y|\vee(-M)\,\de\mu^n(y).
\end{equation*}
A direct computation shows that $\norma{\phi_M^n-\phi^n}_{L^1(B)}\leq \frac\pi2 e^{-2M}$, uniformly with respect to $n$.
Indeed, if $\{a_i^n\}$ is the family of points defining the measure $\mu^n$ in \eqref{average}, by the triangle inequality we have
\[
\begin{split}
\norma{\phi_M^n-\phi^n}_{L^1(B)}&=\frac{1}{n}\int_{B}\Big |\sum_{i=1}^n\log|x-a_i^n|\vee(-M)-\log|x-a_i^n|\Big|\,\de x\\
&\leq -\frac{1}{n}\sum_{i=1}^n \int_{\{x\in \R^2:\ \log|x-a_i|<-M\}}(M+\log|x-a_i^n|)\,\de x.
\end{split}
\]
The latter integrals can be computed in polar coordinates, so that, by taking $\rho=|x - a_i|$,
\[
\norma{\phi_M^n-\phi^n}_{L^1(B)}\leq -2\pi\int_0^{e^{-M}}(M+\log \rho)\rho\,\de \rho=-\pi Me^{-2M}+\pi Me^{-2M}+\frac{\pi}{2}e^{-2M},
\]
which gives the uniform bound claimed above.

Moreover, since for every $x\in B$ the function $\log|x-\cdot|\vee(-M)$ is a continuous function on $\overline\Omega$, by the weak-* convergence $\mu^n\wsto\mu$ we deduce that $\phi^n_M(x)\to\phi_M(x)$ as $n\to\infty$. Then, since $\|\phi^n\|_\infty\leq M\vee |\log (\diam B)|$  by the Dominated Convergence Theorem we obtain $\phi_M^n\to\phi_M$ strongly in $L^1(B)$ as $n\to\infty$.

Finally, since $\mu$ does not charge points, $\log|x-y|\vee(-M)\to \log|x-y|$ as $M\to+\infty$, for all $x\in \B$ and for all $\mu$-a.~e.~$y\in B$. This with the fact $\log|x-y|<\log (\diam \B)$ allows us to use the Monotone Convergence Theorem and obtain $\phi_M(x)\to\phi(x)$ as $M\to+\infty$ for all $x\in \B$. Since also $\phi_M(x)\leq \log(\diam \B)$, by using again the Monotone Convergence Theorem, we obtain that $\phi_M\to\phi$ strongly in $L^1(\B)$ as $M\to+\infty$.

Therefore, by the triangle inequality,
\begin{equation*}
\norma{\phi^n-\phi}_{L^1(B)}\leq \frac\pi2 e^{-2M}
+\norma{\phi_M^n-\phi_M}_{L^1(B)}+\norma{\phi_M-\phi}_{L^1(B)},
\end{equation*}
and the result follows from letting first $n\to\infty$ and then $M\to+\infty$.
\end{proof}

\begin{lemma}\label{fatto1}
Let $y,z\in\B$ with $y\neq z$. There exists a positive constant $C$ independent of $y$ and $z$ such that
\begin{equation}\label{F11}
\int_{\B} \frac1{|x-y|}\frac1{|x-z|}\,\de x \leq{}  C\left ( \int_{\B} \frac{x-y}{|x-y|^2}\cdot\frac{x-z}{|x-z|^2}\,\de x+ 1 \right).
\end{equation}
\end{lemma}
\begin{proof}
Let $2d\coloneqq|y-z|$ and let us considers the two discs $B_{3d}(y)$ and $B_{3d}(z)$ of radius $3d$ centered at $y$ and $z$, respectively.
Let us denote $B_{3d}(y)\cup B_{3d}(z)\eqqcolon D=D^{(y)}\cup D^{(z)}\cup \ell$, where $D^{(y)}$ and $D^{(z)}$ are the two disjoint parts of $D$ on the $y$ and $z$ side of the (open) segment $\ell$ given by the intersection of the axis of $\overline{yz}$ with $D$ {(see Fig. \ref{farla})}.
\begin{figure}[h]
\begin{tikzpicture}
\draw (-1,0) circle (3);
\draw [fill=black] (-1,0) circle (1pt) node [above] {$y$};
\draw (1,0) circle (3);
\draw [fill=black] (1,0) circle (1pt) node [above] {$z$};
\draw[dashed] (0,2.8)--(0,-2.8);
\node at (-0.7,-1.8) {$D^{(y)}$};
\node at (0.7,-1.8) {$D^{(z)}$};
\node at (0.2,1) {$\ell$};
\end{tikzpicture}
\caption{The sets $D^{(y)}$, $D^{(z)}$ and $\ell$.}\label{farla}
\end{figure}
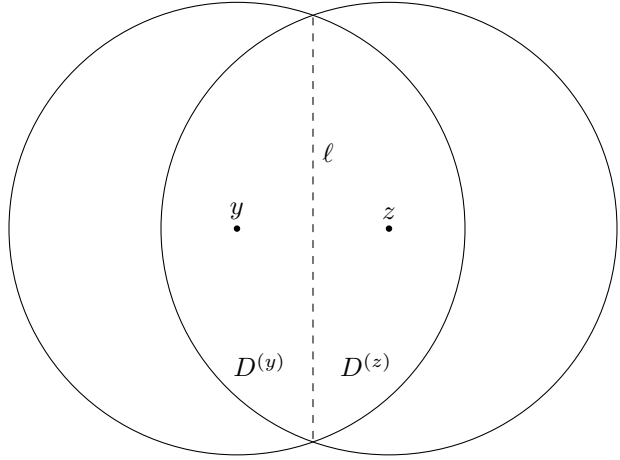
By symmetry of the set $D$, for every $A\subseteq B$ we have
\begin{equation*}
\begin{split}
\int_{A} \frac1{|x-y|}&\frac1{|x-z|}\,\de x \\
{} =& \int_{A\setminus\overline D} \frac1{|x-y|}\frac1{|x-z|}\,\de x + \int_{D^{(y)}\cap A} \frac1{|x-y|}\frac1{|x-z|}\,\de x +\int_{D^{(z)}\cap A} \frac1{|x-y|}\frac1{|x-z|}\,\de x\\
\leq{}& \int_{A\setminus\overline D} \frac1{|x-y|}\frac1{|x-z|}\,\de x + 2\int_{D^{(y)}} \frac1{|x-y|}\frac1{|x-z|}\,\de x \\
\leq{} &  \int_{A\setminus\overline D} \frac1{|x-y|}\frac1{|x-z|}\,\de x + \frac2d\int_{B_{3d}(y)} \frac1{|x-y|}\,\de x.
\end{split}
\end{equation*}
By computing the last integral over $B_{3d}(y)$ in polar coordinates centered at $y$ we obtain
\[
\int_{A} \frac1{|x-y|}\frac1{|x-z|}\,\de x \leq \int_{A\setminus\overline D} \frac1{|x-y|}\frac1{|x-z|}\,\de x + 12\pi.
\]
If we take $A=\B$ in the above inequality, we get
\begin{equation}\label{F12}
\int_{\B} \frac1{|x-y|}\frac1{|x-z|}\,\de x \leq \int_{\B\setminus\overline D} \frac1{|x-y|}\frac1{|x-z|}\,\de x + 12\pi,
\end{equation}
if instead we take $A=B\cap D$, we obtain the estimate
\begin{equation}\label{F15}
\left|\int_{B\cap D} \frac{x-y}{|x-y|^2}\cdot\frac{x-z}{|x-z|^2}\,\de x\right| \leq \int_D \frac1{|x-y|}\frac1{|x-z|}\,\de x\leq 12\pi.
\end{equation}

Moreover, we notice that
\begin{equation}\label{F14}
\frac{x-y}{|x-y|^2}\cdot\frac{x-z}{|x-z|^2}= \frac{\cos\alpha(x)}{|x-y||x-z|},
\end{equation}
where $\alpha(x)$ is the angle centered at $x$ formed by the vectors $z-x$ and $y-x$.
Since there exists an angle $0\leq\alpha_0<\pi/2$ such that $|\alpha(x)|<\alpha_0$ for all $x\in\B\setminus\overline D$, then integrating \eqref{F14} over $\B\setminus\overline D$ leads to
\begin{equation}\label{F16}
\int_{\B\setminus\overline D} \frac{x-y}{|x-y|^2}\cdot\frac{x-z}{|x-z|^2}\,\de x = \int_{\B\setminus\overline D} \frac{\cos\alpha(x)}{|x-y||x-z|}\,\de x \geq \int_{\B\setminus\overline D} \frac{\cos\alpha_0}{|x-y||x-z|}\,\de x.
\end{equation}
Therefore, by combining \eqref{F12} with \eqref{F16}, we obtain
\begin{equation*}
\int_{\B} \frac1{|x-y|}\frac1{|x-z|}\,\de x \leq \frac{1}{\cos \alpha_0}\int_{\B\setminus \overline D} \frac{x-y}{|x-y|^2}\cdot\frac{x-z}{|x-z|^2}\,\de x+12\pi,
\end{equation*}
which, together with \eqref{F15}, gives
\[
\begin{split}
\int_{\B} \frac1{|x-y|}\frac1{|x-z|}\,\de x&\leq \frac{1}{\cos \alpha_0}\int_{\B} \frac{x-y}{|x-y|^2}\cdot\frac{x-z}{|x-z|^2}\,\de x- \int_{B\cap D} \frac{x-y}{|x-y|^2}\cdot\frac{x-z}{|x-z|^2}\,\de x+12\pi\\
&\leq \frac{1}{\cos \alpha_0}\int_{\B} \frac{x-y}{|x-y|^2}\cdot\frac{x-z}{|x-z|^2}\,\de x+24\pi,
\end{split}
\]
which is \eqref{F11} with the constant $C=\max\{1/\cos \alpha_0, 24\pi\}$.
\end{proof}

We prove now some summability properties of the functions appearing in Lemma~\ref{fatto1}.
\begin{lemma}\label{claim}
Let $\mu\in\cP(\overline\Omega)$ be such that
\begin{equation}\label{quantity}
 -\int_{\overline\Om\times\overline\Om}\log|y-z|\de\mu\otimes\mu(y,z)<+\infty.
\end{equation}
Then 
\begin{itemize}
\item[(i)] the functions  $\displaystyle (y,z)\mapsto  \int_{\B} \frac{x-y}{|x-y|^2}\cdot\frac{x-z}{|x-z|^2}\,\de x$ and  $\displaystyle (y,z)\mapsto \int_{\B} \frac1{|x-y|}\frac1{|x-z|}\,\de x $ belong to $L^1(\B\times\B,\mu\otimes\mu)$;
\item[(ii)] the function $\displaystyle x\mapsto \int_{\B} \frac{x-y}{|x-y|^2}\,\de\mu(y)$ belongs to $L^2(\B  ;\mathbb R^2)$.
\end{itemize}
In particular, the function $\phi$ defined in \eqref{T14} belongs to $H^1(\B)$.
\end{lemma}
\begin{proof}
First we notice that \eqref{quantity} yields that $\mu\otimes\mu$ does not charge the diagonal $\triangle_0$ in $\B\times\B$.
Indeed, by \eqref{quantity}, we have that, for $\e>0$ small enough,
\begin{equation*}
+\infty> -\int_{\triangle_\e} \log|y-z|\de\mu\otimes\mu(y,z) \geq -(\log \e) \mu\otimes\mu(\triangle_\e) \geq -(\log\e)\mu\otimes\mu(\triangle_0),
\end{equation*}
where $\triangle_\e\coloneqq\{(y,z)\in\B\times\B:|y-z|< \e\}$ and where we have used the monotonicity of $\mu\otimes\mu$ in the last inequality. 
By taking the limit $\e\to0$ we conclude that $\mu\otimes\mu(\triangle_0)=0$.

For brevity, we set 
\begin{equation*}
\beta(y,z)\coloneqq \int_{\B} \frac{x-y}{|x-y|^2}\cdot\frac{x-z}{|x-z|^2}\,\de x \qquad\text{and}\qquad \gamma(y,z)\coloneqq \int_{\B} \frac1{|x-y|}\frac1{|x-z|}\,\de x.
\end{equation*}
Since $|\beta(y,z)|\leq\gamma(y,z)$ and $\mu$ does not charge $\triangle_0$, we have
\[
\int_{\B\times\B} |\beta(y,z)|  \,\de\mu\otimes\mu(y,z)\leq  \int_{\B\times\B} \gamma(y,z)\,\de\mu\otimes\mu(y,z)={} \lim_{\e\to0} \int_{(\B\times\B)\setminus\overline \triangle_\e} \gamma(y,z)\,\de\mu\otimes\mu(y,z).
\]
The last estimate, together with Lemma~\ref{fatto1} and the fact that $\mu$ concentrates only on $\overline \Omega$,
yields the estimate
\begin{equation}\label{27ago}
\begin{split}
\int_{\B\times\B} |\beta(y,z)|  \,\de\mu\otimes\mu(y,z)\leq{} & C \lim_{\e\to0} \int_{(\B\times\B)\setminus\overline \triangle_\e} \beta(y,z)\,\de\mu\otimes\mu(y,z)+C \\
={} & C\lim_{\e\to0} \int_{(\overline \Omega\times\overline\Omega)\setminus\overline \triangle_\e} \beta(y,z)\,\de\mu\otimes\mu(y,z)+C.
\end{split}
\end{equation}
Taking into account that by the divergence theorem
\begin{equation*}
\beta(y,z) =\int_{\partial\B} \frac{x-y}{|x-y|^2}\cdot \nu(x)\log|x-z|\,\de\cH^1(x)- 2\pi\log|y-z|
\end{equation*}
for $\mu\otimes \mu$-a.~e.~$(y,z)\in \B\times \B$, the estimate \eqref{27ago} becomes
\[
\begin{split}
\int_{\B\times\B} | & \beta(y,z)|  \,\de\mu\otimes\mu(y,z)
\\ & \leq C \int_{\B\times\B}\left(\int_{\partial\B} \frac{x-y}{|x-y|^2}\cdot \nu(x)\log|x-z|\,\de\cH^1(x)- 2\pi\log|y-z|\right)\,
\de\mu\otimes\mu(y,z)+C.
\end{split}
\]
The boundary integral on $\partial B$ can be easily bounded by $\max\{|\log d|, |\log(\diam B)|\}/d$ where $d=\dist(\partial B, \overline \Omega)$, so that the iterated integral over $B\times B$ is finite. Moreover, the term with the logarithm belongs to $L^1(B\times B,\mu\otimes\mu)$ thanks to hypothesis \eqref{quantity}. Therefore (i) follows for the function $\displaystyle (y,z)\mapsto  \int_{\B} \frac{x-y}{|x-y|^2}\cdot\frac{x-z}{|x-z|^2}\,\de x$. The $L^1$ integrability of the function $\displaystyle (y,z)\mapsto \int_{\B} \frac1{|x-y|}\frac1{|x-z|}\,\de x $ follows in a similar way, again by applying Lemma \ref{fatto1}.

To show (ii), we use (i) and the Fubini Theorem \cite[Theorem~8.8 and the following Notes therein]{Rudin} applied to the function $(x,y,z)\mapsto1/(|x-y||x-z|)$ and to the measure spaces $(\B\times\B,\mu\otimes\mu)$ and $(\B,\cL)$. First, we deduce that the iterated integrals 
$$\int_{\B\times\B} \bigg(\int_{\B} \frac1{|x-y|}\frac1{|x-z|}\,\de x\bigg)\de\mu\otimes\mu(y,x) \quad\text{and}\quad \int_{\B} \bigg(\int_{\B\times\B}\frac1{|x-y|}\frac1{|x-z|}\,\de\mu\otimes\mu(y,x)\bigg)\,\de x$$
are finite and equal. This implies that the function $\displaystyle (y,z)\mapsto (x-y)\cdot (x-z)/(|x-y||x-z|)$ belongs to $L^1(\B\times\B,\mu\otimes\mu)$ for a.~e.~$x\in \B$. Similarly, by (i) and Fubini Theorem applied to the function  $(x,y,z)\mapsto(x-y)\cdot (x-z)/(|x-y||x-z|)$, we infer that
$$
\int_{\B\times\B}\bigg(\int_{\B}\frac{x-y}{|x-y|^2}\cdot \frac{x-z}{|x-z|^2}\,\de x\bigg) \de\mu\otimes\mu(y,z)=
\int_{\B}\bigg( \int_{\B\times\B}\frac{x-y}{|x-y|^2}\cdot \frac{x-z}{|x-z|^2}\,\de\mu\otimes\mu(y,z)\bigg)\de x
$$
is finite. In particular, the function $\displaystyle x\mapsto\int_{\B\times\B}\frac{x-y}{|x-y|^2}\cdot \frac{x-z}{|x-z|^2}\,\de\mu\otimes\mu(y,z)$ belongs to $L^1(B)$, so that its value is finite for a.~e.~$x\in\B$. As a consequence, we can apply again Fubini Theorem for a.~e.~$x\in\B$ to the function $\displaystyle (y,z)\mapsto \frac{x-y}{|x-y|^2}\cdot \frac{x-z}{|x-z|^2}\in L^1(\B\times\B,\mu\otimes\mu)$, so we can write
$$\int_{\B}\bigg( \int_{\B\times\B}\frac{x-y}{|x-y|^2}\cdot\frac{x-z}{|x-z|^2}\,\de\mu\otimes\mu(y,z)\bigg)\de x=
\int_{\B}\bigg|\int_{\B} \frac{x-y}{|x-y|^2}\,\de\mu(y)\bigg|^2\,\de x<+\infty.$$
This proves (ii).

It remains to show the $H^1$ regularity of $\phi$: on the one hand, it is easy to see that $\phi$ belongs to $W^{1,p}(\B)$ for every  $1\leq p<2$; on the other hand, a direct computation proves that its distributional gradient agrees with the function defined in (ii). The claim follows by Poincar\'e-Wirtinger inequality combined with the integrability provided in (ii).
This concludes the proof of the lemma.
\end{proof}

\section{Proofs of the main results}\label{sec.4}

We are now ready to prove Theorem~\ref{L1} and Corollaries~\ref{L2} and \ref{L3}.

\begin{proposition}[$\Gamma$-liminf inequality]\label{T30}
Let $n\in \mathbb N$ and $f^n\in H^{-1/2}(\partial \Om)$ with $\langle f^n,1\rangle=2\pi n$. Assume that $\frac1n f^n\to f$ strongly in $H^{-1/2}(\partial \Omega)$ as $n\to\infty$. Then, for every $\mu^n\wsto \mu$ in $\cP(\overline\Omega)$ as $n\to\infty$, there holds
\begin{equation}\label{T5}
\liminf_{n\to\infty} \cF^n(\mu^n)\geq\cF^\infty(\mu).
\end{equation}
\end{proposition}
\begin{proof}
Up to the extraction of a subsequence, we can consider a sequence $\{\mu^n\}$ along which the $\liminf$ in \eqref{T5} is a finite limit. 
This implies that each $\mu^n$ is of the form 
\begin{equation*}
\mu^n=\frac1n\sum_{i=1}^n\delta_{a_i^n},
\end{equation*}
for some distinct points $a_1^n,\ldots,a_n^n\in\Omega$ and, by Lemma~\ref{fatto2}, that $\mu$ does not charge points.

Since the logarithmic term in the right-hand side of \eqref{equic} is bounded below, by Lemma~\ref{lemma-equic} we infer that $\|w^n\|_{H^1(\Omega)} $ and $\|\phi^n\|_{H^1(\B\setminus\overline\Omega)}$ are bounded. Therefore (up to subsequences) $w^n$ converges weakly to some $w\in H^1(\Omega)$ and  $\phi^n$ converges weakly to $\phi\in H^1(\B\setminus\overline\Omega)$, where $\phi$ is the $L^1$-limit found in Lemma \ref{T13}.
By the divergence theorem and \eqref{M4}, recalling \eqref{tracciaesterna}, for every test function $\psi\in H^1(\B\setminus\overline\Omega)$, we have
\begin{equation}\label{tracciaest2}
\langle \partial_\nu\phi^n,\psi\rangle_{\partial \Omega}=- \int_{\B \setminus \overline \Omega} \nabla \phi^n \nabla \psi\, \de x+\int_{\partial \B}\psi\, \partial_\nu \phi^n \,\de\cH^1 \;\stackrel{n\to\infty}{\longrightarrow}\langle \partial_\nu \phi_+ , \psi \rangle_{\partial \Om}.
\end{equation}
In other words, on the boundary $\partial \Omega$, the sequence of normal traces $\partial_\nu \phi^n$ converges to the exterior normal trace $\partial_\nu \phi_+$, weakly in $H^{-1/2}(\partial \Omega)$. These convergences, together with the definition \eqref{T14} of $\phi$, the system \eqref{M4} satisfied by $w^n$, and the assumption on $f^n$, imply that  the limits $\phi$ and $w$ solve
\begin{equation}\label{T19}
\Delta \phi=2\pi\mu\quad  \text{in $B$}
\qquad\text{and}\qquad
\begin{cases}
\Delta w=0 & \text{in $\Omega$,} \\
\partial_\nu w=f-\partial_\nu\phi_+ & \text{on $\partial\Omega$.}
\end{cases}
\end{equation}
Let us first analyze the logarithmic term in the energy: using a truncation argument as in \cite[formula (3.26)]{MPS2017} we can prove that 
\[
 \liminf_{n\rightarrow\infty}-\pi\int_{\Omega\times \Omega} \log|x-y|\,\de\mu^n\boxtimes\mu^n\geq -\pi\int_{\Omega\times \Omega} \log|x-y|\,\de\mu\otimes\mu,
\]
so that the equicoercivity of the energy \eqref{equic} guarantees that the right-hand side above is bounded.
As a consequence, thanks to Lemma~\ref{claim}, the equation for $\phi$ in \eqref{T19} implies in particular that $\mu\in H^{-1}(B)$.

By \eqref{T10}, we can write the energy \eqref{M5} as
\begin{equation*}
\begin{split}
\cF^n(\mu^n)={} &  \frac1{2}\int_{\Omega} |\nabla w^n|^2\,\de x + \frac12\int_{\B\setminus\overline\Omega} |\nabla \phi^n|^2\,\de x -\pi\int_{\Omega\times \Omega} \log|x-y|\,\de\mu^n\boxtimes\mu^n\\
& - \frac12\int_{\partial\B} \phi^n \partial_\nu\phi^n\,\de\cH^1+\frac 1n \langle f^n,\phi^n\rangle_{\partial\Omega},
\end{split}
\end{equation*}
and taking the $\liminf$ as $n\to\infty$, we obtain
\begin{equation}\label{T17}
\begin{split}
\liminf_{n\to\infty} \cF^n(\mu^n) \geq{} &  \frac1{2}\int_{\Omega} |\nabla w|^2\,\de x + \frac12\int_{\B\setminus\overline\Omega} |\nabla \phi|^2\,\de x -\pi\int_{\overline\Omega\times\overline\Omega} \log|x-y|\,\de\mu\otimes\mu\\
& - \frac12\int_{\partial\B} \phi\, \partial_\nu\phi\,\de\cH^1+\langle f,\phi\rangle_{\partial\Om}.
\end{split}
\end{equation}
Here we have used the assumption $\frac 1n f^n\to f$ strongly in $H^{-1/2}(\partial \Omega)$, the lower semicontinuity of the $H^1$ norm, and that of the term with the logarithm.
Recalling \eqref{T14} and \eqref{T19},  integrating by parts the term with the logarithm we have
\begin{equation}\label{T18}
\begin{split}
-\pi\int_{\overline\Omega\times\overline\Omega} \log|x-y|\,\de\mu\otimes\mu & = -\pi\int_{\B\times\B} \log|x-y|\,\de\mu\otimes\mu \\
&=\frac12\int_{\B} |\nabla\phi|^2\,\de x-\frac12\int_{\partial\B} \phi\,\partial_\nu\phi_+\,\de\cH^1 \\
&=\frac12\int_{\Omega} |\nabla\phi|^2\,\de x+\frac12\int_{\B\setminus\Omega} |\nabla\phi|^2\,\de x-\frac12\int_{\partial\B} \phi\,\partial_\nu\phi_+\,\de\cH^1 \\
&=\frac12\int_{\Omega} |\nabla\phi|^2\,\de x-\frac12\langle \partial_\nu\phi_+,\phi\rangle_{\partial\Omega},
\end{split}
\end{equation}
where in the last equality we have used \eqref{tracciaest2}.
Moreover, integrating by parts in $\B\setminus\overline\Omega$, we have
\begin{equation}\label{T20}
\frac12\int_{\B\setminus\overline\Omega} |\nabla \phi|^2\,\de x- \frac12\int_{\partial\B} \phi\, \partial_\nu\phi\,\de\cH^1=- \frac12\langle \partial_\nu\phi_+ , \phi \rangle_{\partial\Omega},
\end{equation}
so that, by adding \eqref{T18} and \eqref{T20}, the right-hand side of \eqref{T17} becomes
\begin{equation*}
 \frac1{2}\int_{\Omega} |\nabla w|^2\,\de x +\frac12\int_{\Omega} |\nabla\phi|^2\,\de x+\langle f-\partial_\nu\phi_+,\phi\rangle_{\partial\Omega}.
\end{equation*}
Defining $U\coloneqq w+\phi$ and using \eqref{T19} and Remark~\ref{allafinelocitiamo}, we have that $U$ solves \eqref{T2} and the expression above is precisely the functional $\cF^\infty(\mu)$ defined in \eqref{T1}.
This concludes the proof of \eqref{T5} and of the proposition.
\end{proof}

\begin{proposition}[$\Gamma$-limsup inequality]\label{T3a}
Let $n\in \mathbb N$ and $f^n\in H^{-1/2}(\partial \Om)$ with $\langle f^n,1\rangle=2\pi n$. Assume that $\frac1n f^n\to f$ strongly in $H^{-1/2}(\partial \Omega)$ as $n\to\infty$.
Let $\mu\in \cP(\overline\Omega)$.
Then there exists a sequence of measures $\mu^n\in \cP(\overline\Omega)$ such that $\mu^n\wsto\mu$ in $\cP(\overline\Omega)$ as $n\to\infty$ and  
\begin{equation}\label{T5a}
\limsup_{n\to\infty} \cF^n(\mu^n)\leq\cF^\infty(\mu).
\end{equation}
\end{proposition}
\begin{proof}
In the case $\mu\notin H^{-1}(B)$ the inequality is trivially satisfied by choosing $\mu^n=\mu$. 
Therefore, let us assume that $\mu\in H^{-1}(B)$.
Since $\mathcal F^\infty(\mu)<+\infty$, in order to prove \eqref{T5a} we look for a sequence of approximating measures $\mu^n$ of the form \eqref{average}, for $n$ distinct points $a^n_1,\ldots,a_n^n\in\Omega$, which implies that $\mathcal F^n(\mu^n)$ is finite and can be written as in \eqref{M5}.

\textbf{\emph{Step 1}.} We first prove the $\limsup$ inequality assuming that $\mu$ is absolutely continuous with respect to the Lebesgue measure and is of the form
\begin{equation}\label{T22}
\de\mu=\sum_{j}\alpha_j\1_{\Omega_j}\de x,
\end{equation}
where $\Omega_j$ is a finite family of pairwise disjoint Borel sets defined as follows. For a fixed parameter $h>0$, we define $Q_j$ as the collection of all open squares with corners on the lattice $h\mathbb Z^2$ such that $\overline{Q}_j \cap \Omega \neq \emptyset$. Denoting by $\Gamma_j$ the closure of the right and top sides of $\partial Q_j$, we set (see Figure~\ref{quadratelli})
\begin{equation}\label{22ago}
 \Omega_j\coloneqq (\overline{Q}_j \cap \overline\Omega) \setminus (\Gamma_j\setminus\partial\Omega).
\end{equation}
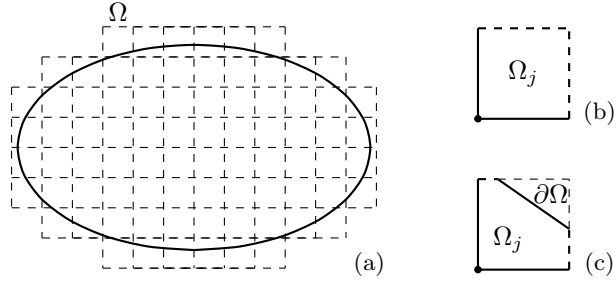
\begin{figure}[h]
\begin{tikzpicture}[scale=0.4]

\draw[thick, domain=-3.141: 3.141, smooth]plot({5.8*sin(\x r)}, {3.4*cos(\x r)}) node at (-2.5,4.5) {$\Omega$};
\coordinate (a1) at (-6,2);
\coordinate (a2) at (-5,3);
\coordinate (a3) at (-4,3);
\coordinate (a4) at (-3,4);
\coordinate (a5) at (-2,4);
\coordinate (a6) at (-1,4);
\coordinate (a7) at (0,4);
\coordinate (a8) at (1,4);
\coordinate (a9) at (2,4);
\coordinate (a10) at (3,4);
\coordinate (a11) at (4,3);
\coordinate (a12) at (5,3);
\coordinate (a13) at (6,2);

\coordinate (b1) at (-6,-2);
\coordinate (b2) at (-5,-3);
\coordinate (b3) at (-4,-3);
\coordinate (b4) at (-3,-4);
\coordinate (b5) at (-2,-4);
\coordinate (b6) at (-1,-4);
\coordinate (b7) at (0,-4);
\coordinate (b8) at (1,-4);
\coordinate (b9) at (2,-4);
\coordinate (b10) at (3,-4);
\coordinate (b11) at (4,-3);
\coordinate (b12) at (5,-3);
\coordinate (b13) at (6,-2);

\draw [dashed] (a1) -- (b1);
\draw [dashed] (a2) -- (b2);
\draw [dashed] (a3) -- (b3);
\draw [dashed] (a4) -- (b4);
\draw [dashed] (a5) -- (b5);
\draw [dashed] (a6) -- (b6);
\draw [dashed] (a7) -- (b7);
\draw [dashed] (a8) -- (b8);
\draw [dashed] (a9) -- (b9);
\draw [dashed] (a10) -- (b10);
\draw [dashed] (a11) -- (b11);
\draw [dashed] (a12) -- (b12);
\draw [dashed] (a13) -- (b13);

\draw [dashed] (a6) -- (a8);
\draw [dashed] (a5) -- (a9);
\draw [dashed] (a4) -- (a10);
\draw [dashed] (a3) -- (a11);
\draw [dashed] (a2) -- (a12);
\draw [dashed] (a1) -- (a13);
\draw [dashed] (-6,1) -- (6,1);
\draw [dashed] (-6,0) -- (6,0);
\draw [dashed] (-6,-1) -- (6,-1);
\draw [dashed] (b6) -- (b8);
\draw [dashed] (b5) -- (b9);
\draw [dashed] (b4) -- (b10);
\draw [dashed] (b3) -- (b11);
\draw [dashed] (b2) -- (b12);
\draw [dashed] (b1) -- (b13);

\coordinate (L) at (5.77,-3.85) node at (5.77,-3.85) {\small (a)};

\end{tikzpicture} 
\ \ \ \ \  \ \ \ 
\begin{tikzpicture}[scale=0.4]
\draw[thick] (0,5) -- (3,5);
\draw [thick] (0,5) -- (0,8);
\draw[dashed,thick] (3,5) -- (3,8);
\draw [dashed,thick] (0,8) -- (3,8);
\draw [fill=black] (0,5) circle (3pt);
\coordinate (Rb) at (4,5.2) node at (4,5.2) {\small (b)};
\coordinate (Rc) at (4,0.2) node at (4,0.2) {\small (c)};
\coordinate (E1) at (1,1) node at (1,1) {$\Omega_j$};
\coordinate (E2) at (2.4,2.5) node at (2.4,2.5) {$\partial\Omega$};
\coordinate (E3) at (1.5,6.5) node at (1.5,6.5) {$\Omega_j$};
\coordinate (z1) at (0.65,3);
\coordinate (z2) at (3,1.35);
\draw[thick] (0,0) -- (3,0);
\draw [thick] (0,0) -- (0,3);
\draw[dashed,thick] (3,0) -- (z2);
\draw[dashed,thin] (z2) -- (3,3);
\draw [dashed,thick] (0,3) -- (z1);
\draw [dashed,thin] (z1) -- (3,3);
\draw [thick] (z1) -- (z2);
\draw [fill=black] (0,0) circle (3pt);
\end{tikzpicture} 
\caption{(a) Dashed, the family of open squares $Q_j$. (b) Definition of the Borel sets $\Omega_j$ when $\overline{Q}_j\cap \partial \Omega = \emptyset$. (c)  Definition of the Borel sets $\Omega_j$ when $\overline{Q}_j\cap \partial \Omega \neq \emptyset$.}\label{quadratelli}
\end{figure}

The coefficients $\alpha_j$ in \eqref{T22} are such that $\alpha_j\in[0,1]$ for all $j$, $\alpha_j=0$ whenever $\overline\Omega_j\cap \partial \Omega\neq \emptyset$, and  $\sum_j \alpha_j|\Omega_j|=1$.
With this choice of the $\alpha_j$'s, the measure $\mu$ is a piecewise constant probability measure with compact support in $\Omega$.

For every $n$, we construct $\mu^n\coloneqq\frac1n\sum_{i=1}^n\delta_{a_i^n}$, for some distinct points $a_1^n,\ldots,a_n^n\in\Omega$, such that $\mu^n\wsto\mu$, by placing a suitable number of dislocations of the order of $\lfloor n\alpha_j|\Omega_j|\rfloor$ in each $\Omega_j$ (this is a standard construction, see, e.g., \cite{BJM2011,MPS2017} and also \cite{BS2007,TZ2013,TZ2015}). 
Since the dislocations remain at positive distance to $\partial \Omega$ we have that the functions 
$\phi^n$ associated with $\mu^n$ via \eqref{M4} converge strongly in $H^1(\B\setminus\overline\Omega)$ to the function $\phi$ defined in \eqref{T14};  hence the first three terms in \eqref{M5} converge. Moreover, by \cite[formula~(3.32)]{MPS2017} it follows that
\[
\limsup_j -\pi\int_{\Omega\times\Omega} \log|x-y|\,\de\mu^n\boxtimes\mu^n\leq -\pi\int_{\overline\Omega\times\overline\Omega} \log|x-y|\,\de\mu\otimes\mu.
\]
Therefore, \eqref{T5a} is proved when $\mu$ is a piecewise constant measure with compact support in $\Omega$.

\textbf{\emph{Step 2}.} By \cite{dalmaso}, to obtain \eqref{T5a} for a general $\mu\in H^{-1}(B)$, it is enough to prove that the class of piecewise constant measures with compact support in $\Omega$ is \emph{dense in energy}, namely that there exists a sequence of piecewise constant measures $\mu_h\in\cP(\overline\Omega)$ with compact support in $\Omega$ such that $\mu_h\wsto\mu$ and 
\begin{equation}\label{T5b}
\limsup_{h\to\infty} \cF^\infty(\mu_h)\leq\cF^\infty(\mu).
\end{equation}
We construct the approximating measures $\mu_h$ of the form \eqref{T22}, by choosing coefficients $\alpha_j^h$ and sets $\Omega_j^h$ as follows.
We consider a parameter $h>0$ and we define the collection $\{Q_j^h\}_j$ of open squares intersecting $\Omega$ as in Step 1. Accordingly, we define the $\Gamma_j^h$'s, and the $\Omega_j^h$'s as in \eqref{22ago}, namely $\Omega_j^h\coloneqq (\overline{Q}{}_j^h \cap \overline\Omega) \setminus (\Gamma_j^h\setminus\partial\Omega)$.
We observe that, for every $h>0$, $\Omega_j^h\cap\Omega_k^h=\emptyset$ if $j\neq k$, that $\overline\Omega=\bigcup_j \Omega_j^h$, and that for any $j$ we have $|\Omega_j^h|>0$.

For every $j$, we define $\beta_j^h:=\mu(\Omega_j^h)/|\Omega_j^h|$,
so that the approximating measure $\mu_h$ of the form \eqref{T22} is constructed as follows: for each $j$
\begin{itemize}
\item[(i)] if $\Omega_j^h\cap \partial \Omega=\emptyset$ then we set $\alpha_j^h:=\beta_j^h$.
\item[(ii)] otherwise, if $\Omega_j^h\cap \partial \Omega\neq \emptyset$ we set $\alpha_j^h=0$ and transfer the mass $\beta_j^h$ to $\Omega_K^h$, where $\Omega_K^h$ is a square such that $\Omega_K^h\cap \partial\Omega=\emptyset$ and is the closest to $\Omega_j^h$ (clearly, there can be more than one such $\Omega_K^h$), that is
$$\alpha_K^h\coloneqq \beta_K^h+\beta_j^h\frac{|\Omega_j^h|}{|\Omega_K^j|}.$$
\end{itemize}
We finally define $\mu^h$ as $\de\mu^h\coloneqq \sum_j \alpha_j^h \1_{\Omega_j^h}\de x$ and notice that it has compact support in $\Omega$, as desired.

We claim that these measures $\mu_h$ converge strongly to $\mu$ in $H^{-1}(B)$.
We first observe that $2\pi\mu_h(x)=\Delta \phi_h$, where $\phi_h(x)=\int_B \log|x-y|\,\de\mu_h(y)$, and similarly $2\pi\mu=\Delta\phi$, where $\phi$ is defined in \eqref{lorecolliamo}.
By using the definition of the $H^{-1}$ norm, proving the convergence of $\mu_h$ to $\mu$ is equivalent to proving that $\nabla\phi_h\to\nabla\phi$ strongly in $L^2(B;\R^{2})$. 
Since
$$\int_B |\nabla \phi_h|^2\,\de x = \int_{\partial B} \phi_h\,\partial_\nu\phi_h\,\de\cH^1-2\pi\int_{B\times B} \log|x-y|\,\de\mu_h\otimes\mu_h$$
and the boundary term converges to $\int_{\partial B} \phi\,\partial_\nu\phi\,\de\cH^1$, invoking the lower semincontinuity of the $L^2$ norm, we are left with proving that 
$$\limsup_{h\to\infty} -\int_{B\times B} \log|x-y|\,\de\mu_h\otimes\mu_h\leq -\int_{B\times B} \log|x-y|\,\de\mu\otimes\mu.$$
This can be proved using a truncation argument as in \cite[equation (3.27)]{MPS2017}, so that we obtain 
\begin{equation}\label{numeraquellaformula}
\norma{\nabla \phi_h}_{L^2(B;\R^2)}\to \norma{\nabla \phi}_{L^2(B;\R^2)}.
\end{equation}
The boundedness of $\norma{\nabla \phi_h}_{L^2(B;\R^{2})}$ implies that $\nabla \phi_h$ converges weakly in $L^2(B;\R^2)$ to its distributional limit $\nabla \phi$, so that \eqref{numeraquellaformula} allows us to conclude that the convergence of $\nabla\phi_h$ to $\nabla\phi$ is indeed strong in $L^2(B;\R^{2})$.

Thus the solutions $U_h$ associated with $\mu_h$ by \eqref{T2} converge strongly in $H^1(\Omega)$ to the solution $U$ associated with $\mu$.
Then \eqref{T5b} follows by definition of $\cF^\infty$.
\end{proof}

\begin{proof}[Proof of Theorem \ref{L1}]
Propositions~\ref{T30} and~\ref{T3a} imply that the functionals $\mathcal{F}^n$ defined in \eqref{T9} $\Gamma$-converge to the functional $\mathcal{F}^\infty$ defined in \eqref{T1} (see \cite{dalmaso}).
Theorem~\ref{L1} is then proved. 
\end{proof}
\begin{proof}[Proof of Corollary~\ref{L2}]
The proof follows from \cite[Corollary~7.20]{dalmaso}. 
The strict convexity of $\mathcal{F}^\infty$ implies that its minimizer is unique (and thus the convergence holds without extracting subsequences).
\end{proof}
\begin{proof}[Proof of Corollary~\ref{L3}]
The uniqueness of the minimizer of $\mathcal{F}^\infty$ obtained in Corollary~\ref{L2} and the fact that $\mathcal{F}^\infty(\mu)\geq 0$ for any $\mu\in\mathcal P(\overline\Omega)$ imply that it is enough to show that the measure $\mu^\infty$ in \eqref{limite} is such that $\mathcal{F}^\infty(\mu^\infty)=0$.
This assertion is a simple verification obtained by integrating by parts \eqref{T1} and using \eqref{T2}.
The limit \eqref{vanishinglimit} is granted again by Corollary~\ref{L2}.
\end{proof}

\subsection*{Acknowledgements}
The authors are members of the Gruppo Nazionale per l'Analisi Matematica, la Probabilit\`a e le loro Applicazioni (GNAMPA) of the Istituto Nazionale di Alta Matematica (INdAM).
The INdAM-GNAMPA project 2015 \href{http://fcm2.weebly.com/}{\emph{Fenomeni Critici nella Meccanica dei Materiali: un Approccio Variazionale}} partially supported this research.
M.M.\@ acknowledges partial support from the DFG Project \emph{Identifikation von Energien durch Beobachtung der zeitlichen Entwicklung von Systemen} (FO 767/7).
R.S.\@  is grateful to the Erwin  Schr\"odinger Institute for the financial support obtained during the last part of the present research.
D.Z.\@ acknowledges partial support from the INdAM-GNAMPA
project 2018 \emph{Ottimizzazione Geometrica e Spettrale}.


\begin{thebibliography}{99}

\bibitem{ADLGP} {\sc R.\@ Alicandro, L.\@ De Luca, A.\@ Garroni, and M.\@ Ponsiglione}: \emph{Metastability and Dynamics of Discrete Topological Singularities in Two Dimensions: A  $\Gamma$-Convergence Approach}, Arch. Rational Mech. Anal., {214}, 269--330, 2014.

\bibitem{BBH} {\sc F.\@ Bethuel, H.\@ Brezis, and F.\@ H\'elein}: Ginzburg-{L}andau vortices, Progress in Nonlinear Differential Equations and their Applications, \textbf{13}, Birkh\"auser Boston, Inc., Boston, MA, 1994.

\bibitem{BFLM} {\sc T.\@ Blass, I.\@ Fonseca, G.\@ Leoni, and M.\@ Morandotti}: \emph{Dynamics for systems of screw dislocations}. SIAM J. Appl. Math., \textbf{75}  (2015), 393--419.

\bibitem{BM} {\sc T.\@ Blass and M.\@ Morandotti}: \emph{Renormalized energy and Peach-K\"ohler forces for screw dislocations with antiplane shear}. J. Convex Anal., \textbf{24}(2) (2017), 547--570.

\bibitem{BJM2011} {\sc G.\@ Bouchitt\'e, C.\@ Jimenez, and R.\@ Mahadevan}: \emph{Asymptotic analysis of a class of optimal location problems}. J. Math. Pures Appl. (9) \textbf{95}(4) (2011), 382--419.

\bibitem{BS2007} {\sc G.\@ Buttazzo and F.\@ Santambrogio}: \emph{Asymptotical compliance optimization for connected networks}. Netw. Heterog. Media, \textbf{2} (2007), 761--777.

\bibitem{CG} {\sc P.\@ Cermelli and M.\@ E.\@ Gurtin}: \emph{The motion of screw dislocations in crystalline materials undergoing antiplane shear: glide, cross-slip, fine cross-slip}. Arch. Ration. Mech. Anal., \textbf{148} (1999), 3--52.

\bibitem{CL} {\sc P.\@ Cermelli and G.\@ Leoni}: \emph{Renormalized energy and forces on dislocations}. SIAM J. Math. Anal., \textbf{37} (2005), 1131--1160.

\bibitem{ChenFrid} {\sc G. Q.\@ Chen and H.\@ Frid}: \emph{On The Theory Of Divergence-Measure Fields And Its Applications}. Bol.\@ Soc.\@ Bras.\@ Mat., Vol. \textbf{32}(3) (2001), 401--433.

\bibitem{dalmaso} {\sc G.\@ {Dal Maso}}, {An introduction to {$\Gamma$}-convergence}. Progress in Nonlinear Differential Equations and their Applications, \textbf{8}, Birkh{\"a}user Boston Inc., Boston, MA, 1993.

\bibitem{GLP} {\sc A.\@ Garroni, G.\@ Leoni, and M.\@ Ponsiglione}: \emph{Gradient theory for plasticity via homogenization of discrete dislocations}. J. Eur. Math. Soc., \textbf{12} (2010), 1231--1266.

\bibitem{GSV} {\sc E.\@ C.\@ Gartland Jr., A.\@ M.\@ Sonnet, and E.\@ G.\@ Virga}: \emph{Elastic forces on nematic point defects}. Contin.\@ Mech.\@ Thermodyn.\@ \textbf{14} (2002), 307--319.

\bibitem{GPPS2013} {\sc M.\@ G.\@ D.\@ Geers, R.\@ H.\@ J.\@ Peerlings, M.\@ A.\@ Peletier, and L.\@ Scardia}: \emph{Asymptotic Behaviour of a Pile-Up of Infinite Walls of Edge Dislocations}. Arch.\@ Rational Mech.\@ Anal., \textbf{209} (2013), 495--539.

\bibitem{Had} {\sc J.\@ Hadamard}, {Lectures on Cauchy's problem in linear partial differential equations}. Dover Phoenix editions, Dover Publications, New York, 1923.

\bibitem{HHW} {\sc P.\@ B.\@ Hirsch, R.\@ W.\@ Horne, and M.\@ J.\@ Whelan}: \emph{Direct observations of the arrangement and motion of dislocations in aluminium}. Phil. Mag \textbf{1} (1956), 677--684.

\bibitem{HL} {\sc J.\@ P.\@ Hirth and J.\@ Lothe}: {Theory of Dislocations}. Krieger Publishing Company, 1992.

\bibitem{HM} {\sc T.\@ Hudson and M.\@ Morandotti}: \emph{Properties of screw dislocation dynamics: time estimates on boundary and interior collisions}. SIAM J. Appl. Math., \textbf{77}(5) (2017), 1678--1705.

\bibitem{HB} {\sc D.\@ Hull and D.\@ J.\@ Bacon}. Introduction to dislocations. Butterworth-Heinemann, 2001.

\bibitem{LMSZ} {\sc I.\@ Lucardesi, M.\@ Morandotti, R.\@ Scala, and D.\@ Zucco}: \emph{Confinement of dislocations inside a crystal with a prescribed external strain}. arXiv:1610.06852. \emph{Submitted}.

\bibitem{vMM2018} {\sc P.\@ van Meurs and M.\@ Morandotti}: \emph{Discrete-to-continuum limits of particles with an annihilation rule}. {arXiv:1807.11199}. \emph{Submitted}.

\bibitem{MPS2017} {\sc M.\@ G.\@ Mora, M.\@ A.\@ Peletier, and L.\@ Scardia}: \emph{Convergence of interaction-driven evolutions of dislocations with Wasserstein dissipation and slip-plane confinement}. SIAM J.\@ Math.\@ Anal. \textbf{49}(5) (2017), 4149--4205.

\bibitem{nabarro} {\sc F.\@ R.\@ N.\@ Nabarro}: {Theory of crystal dislocations}. International series of monographs on physics. Clarendon P., 1967.

\bibitem{orowan} {\sc E.\@ Orowan}: \emph{Zur kristallplastizit\"at. III}. Zeitschrift f\"ur Physik \textbf{89} (1934), 634--659.

\bibitem{polanyi} {\sc M.\@ Polanyi}: \emph{\"Uber eine art gitterst\"orung, die einen kristall plastisch machen k\"onnte}. Zeitschrift f\"ur Physik \textbf{89} (1934), 660--664.

\bibitem{pons} {\sc M.\@ Ponsiglione}: \emph{Elastic energy stored in a crystal induced by screw dislocations: from discrete to continuous}. SIAM J. Math. Anal., \textbf{39}(2) (2007), 449--469.

\bibitem{Rudin} {\sc W.\@ Rudin}. Real and complex analysis. Third edition. McGraw-Hill Book Co., New York, 1987.

\bibitem{SaSe} {\sc E.\@ Sandier and S.\@ Serfaty}, Vortices in the magnetic Ginzburg-Landau model, Progress in Nonlinear Differential Equations and their Applications, 70. Birkh\"auser Boston, Inc., Boston, MA, 2007.


\bibitem{SS2000}  {\sc E.\@ Sandier and M\@ Soret}: \emph{{$\mathbb{S}^1$}-Valued Harmonic Maps with High Topological Degree}. Harmonic morphisms, harmonic maps, and related topics, 141--145, Chapman \& Hall/CRC Res. Notes Math., 413, Chapman \& Hall/CRC, Boca Raton, FL, 2000.

\bibitem{SS00} {\sc E.\@ Sandier and M\@ Soret}: \emph{{$\mathbb{S}^1$}-Valued Harmonic Maps with High Topological Degree: Asymptotic Behavior of the Singular Set}. Potential Anal.\@ \textbf{13}(2) (2000), 169--184.

\bibitem{TOP} {\sc E.\@ B.\@ Tadmor, M.\@ Ortiz, and R.\@ Phillips}: \emph{Quasicontinuum analysis of defects in solids}. Philosophical Magazine A, \textbf{73} (1996), 1529--1563.

\bibitem{taylor} {\sc  G.\@ I.\@ Taylor}: \emph{The Mechanism of Plastic Deformation of Crystals. Part I. Theoretical}. Proceedings of the Royal Society of London. Series A \textbf{145} (855)  (1934), 362--387.

\bibitem{TZ2013} {\sc P.\@ Tilli and D.\@ Zucco}: \emph{Asymptotics of the first Laplace eigenvalue with Dirichlet regions of prescribed length}. SIAM J. Math. Anal. {\bf 45} (2013), 3266--3282.

\bibitem{TZ2015} {\sc P.\@ Tilli and D.\@ Zucco}, \emph{Where best to place a Dirichlet condition in an anisotropic membrane?} SIAM J. Math. Anal. {\bf 47} (2015), 2699--2721.

\bibitem{TZ2018} {\sc P.\@ Tilli and D.\@ Zucco}, \emph{Spectral partitions for Sturm-Liouville problems}. {arxiv:1807.05973}. \emph{Submitted}

\bibitem{VKLLO} {\sc B.\@ Van Koten, X.\@ H.\@ Li, M.\@ Luskin, and C.\@ Ortner}: \emph{A computational and theoretical investigation of the accuracy of quasicontinuum methods}, in Numerical Analysis of Multiscale Problems, Lect. Notes Comput. Sci. Eng. \textbf{83} (2012), 67--96.

\bibitem{volterra} {\sc  V.\@ Volterra}: \emph{Sur l'\'equilibre des corps \'elastiques multiplement connexes}. Annales scientifiques de l'\'Ecole Normale Sup\'erieure, \textbf{24} (1907), 401--517.

\end{thebibliography}
\end{document}